\documentclass{amsart}

\usepackage{graphicx}
\usepackage{amssymb}

\usepackage[OT2,T1]{fontenc}

\usepackage{hyperref}

\usepackage{color}

\usepackage{ulem} 

\newtheorem{thm}{Theorem}[section]
\newtheorem{cor}[thm]{Corollary}
\newtheorem{lem}[thm]{Lemma}
\newtheorem{prop}[thm]{Proposition}
\newtheorem*{obser}{Observation}
\newtheorem{problem}[thm]{Problem}
\theoremstyle{definition}
\newtheorem{defn}[thm]{Definition}

\newtheorem{ex}[thm]{Example}
\theoremstyle{remark}
\newtheorem{rem}[thm]{Remark}

\newcommand{\n}[1]{\left\Vert#1\right\Vert}

\newcommand{\K}{\mathbb{K}}
\newcommand{\R}{\mathbb{R}}
\newcommand{\C}{\mathbb{C}}

\newcommand{\Marku}{Markushevich}

\DeclareMathOperator{\spann}{span}
\DeclareMathOperator{\ran}{ran}

\DeclareMathOperator{\dens}{dens}

\begin{document}
\renewcommand{\proofname}{\scshape Proof}

\title[Totally smooth renormings]{Totally smooth renormings}%
\author{Eve Oja,  Tauri Viil, and Dirk Werner}%
\address{}%
\email{}%

\address[Oja]{Institute of Mathematics and Statistics, University of Tartu, J.~Liivi~2, 50409~Tartu, Estonia, and Estonian Academy of Sciences, Kohtu~6, 10130~Tallinn, Estonia}

\email{eve.oja@ut.ee}

\address[Viil]{Institute of Mathematics and Statistics, University of Tartu, J.~Liivi~2, 50409~Tartu, Estonia}

\email{tauriv@ut.ee}

\address[Werner]{Department of Mathematics \\ Freie Universit\"at Berlin \\ 
Arnimallee~6 \\ \newline 14195~Berlin \\ Germany\newline
\href{http://orcid.org/0000-0003-0386-9652}{ORCID: \texttt{0000-0003-0386-9652}}
}
\email{werner@math.fu-berlin.de}

\thanks{}%

\date{17.7.2018}%

\begin{abstract}
  We study the problem of totally smooth renormings of Banach spaces and provide such renormings for spaces which are weakly compactly generated. We also consider renormings for $(a,B,c)$-ideals. 
\end{abstract}
\maketitle

\section{Introduction}
Let $X$ be a Banach space. Following Phelps \cite{P}, we say that a subspace $X$ of a Banach space $Z$ has \emph{property~$U$} in $Z$ if every functional $x^* \in X^*$ has a unique norm-preserving extension $z^* \in Z^*$. Following Liao and Wong \cite{LW}, we say that $X$ is \emph{totally smooth} in $Z$ if every closed subspace $Y$ of $X$ has property~$U$ in $Z$. If $Z=X^{**}$, then we say that $X$ has property~$U$ in its bidual or, respectively, that $X$ is totally smooth in its bidual.
Banach spaces with property~$U$ in their biduals are also known as \textit{Hahn-Banach smooth} spaces \cite{S}.

The notion of total smoothness in the bidual was essentially considered already in 1977 by Sullivan \cite{S}.
It was further studied in \cite{OPV}, where several geometrical conditions equivalent to total smoothness were proved.

Let $\pi_X \colon X^{***} \to  X^{***}$ denote the natural projection onto the dual space $X^*$. It is known (see \cite{O2} or, e.g., \cite[p.~21]{O4}) that $X$ has property~$U$ in its bidual $X^{**}$ if and only if $X$ has the strong uniqueness property $SU$ in $X^{**}$, meaning that the following condition holds: for $x^{***}\in X^{***}$,
\[
\|\pi_X x^{***}\| = \|x^{***}\| \quad\Rightarrow \quad \pi_X x^{***} = x^{***}.
\]
 
A Banach space $X$ is called an \emph{$M$-ideal} in its bidual if the equality
\[
\n{x^{***}}=\n{\pi_X x^{***}}+\n{x^{***}-\pi_X x^{***}}
\]
holds for every $x^{***} \in X^{***}$. 
The notion of $M$-ideals was introduced by Alfsen and Effros in \cite{AE} and has since then been studied by many authors (see, e.g., the monograph \cite{HWW} for results and references).


The motivation for this paper comes from the following observation in \cite[Remark~2.8]{OPV}, which is based on \cite[Theorem~III.4.6]{HWW}.

\begin{obser}\label{observ}
If $X$ is an $M$-ideal in its bidual $X^{**}$, then $X$ admits an equivalent norm under which $X$ is totally smooth and is still an $M$-ideal in its bidual. 
\end{obser}

As, clearly, being an $M$-ideal implies property~$U$, 
the natural question arises whether the $M$-ideal condition in the Observation could be relaxed.

\begin{problem} \label{problem: U}
If $X$ has property~$U$ in its bidual $X^{**}$, then does $X$ admit an equivalent norm under which $X$ is totally smooth in its bidual?
\end{problem}

Let us note at once (see Example~\ref{ex: l_p sum c_0} below) that there exist 
Banach spaces that admit an equivalent norm under which they are totally smooth, but they do not admit any equivalent norm under which they are $M$-ideals in their biduals.

Now, we recall that according to the Taylor--Foguel theorem (see \cite{T} and \cite{Fo}), \textit{every subspace $Y$ of $X$ has property~$U$ in $X$ if and only if the dual space $X^*$ is strictly convex,}  i.e., its unit sphere $S_{X^*}$ contains no non-trivial line segments.

Thus, relying on the Taylor--Foguel theorem, we can characterize totally smooth spaces by the strict convexity of the dual space as follows.

\begin{thm}[see \cite{LW}]\label{thm: totally smooth}
A Banach space $X$ is totally smooth in its bidual $X^{**}$ if and only if $X$ has property~$U$ in $X^{**}$ and the dual space $X^*$ is strictly convex.
\end{thm}

Therefore, Problem~\ref{problem: U} is equivalent to the following problem. 

\begin{problem} \label{problem: U convex}
If $X$ has property~$U$ in its bidual $X^{**}$, then does $X$ admit an equivalent norm under which the dual space $X^*$ is strictly convex and $X$ still has property~$U$ in its bidual?
\end{problem}




Recall that if $X$ is separable and has property $U$ in $X^{**}$, then $X^*$ is 
separable (see Theorem~\ref{Asplund} below). It was observed by Sullivan in \cite[p.~321]{S} that an application of the
Kadets--Klee renorming
theorem solves Problem~\ref{problem: U convex} (and thus also
Problem~\ref{problem: U}) fully and positively for separable spaces. 
This theorem provides a Banach space having a
separable dual with an equivalent norm whose dual norm is strictly convex and
has the  property that the weak$^*$ topology and the norm
topology coincide on the (new) dual sphere.
A proof of this renorming
theorem can be found in \cite[pp.~113--117]{Di}; it first appeared in Klee's paper
\cite{Kl2} relying on work by Kadets in \cite{Ka}.

In Section~3, using a simple Klee-type renorming \cite{Kl1}, 
we give a
partial positive answer to Problem~\ref{problem: U convex}, and thus to Problem~\ref{problem: U}, in the general case (see Theorem~\ref{main theorem prop U}).
In particular, our results also provide an alternative proof for the separable case. In Section~4, we come back to the Observation and we show that, under natural assumptions, its claim holds for $(a,B,c)$-ideals, which are a far-reaching generalization of $M$-ideals, encompassing 
in particular
$u$- and $h$-ideals.

Our notation is standard. We consider Banach spaces over the scalar field $\K=\R$ or $\K=\C$. For a Banach space $X$, $B_X$ is the closed unit ball and $S_X$ is the unit sphere of $X$. By $\overline{\spann}(x_i)$, we denote the closed linear span of the elements $x_i$. For a subspace $X$ of $Z$, $X^\perp=\{z^*\in Z^*\colon {z^*}_{|X}=0\}$ is the annihilator. 
The density character of the space $X$ is denoted by $\dens X$. 
For a bounded linear
operator $T$, $T^*$ is the adjoint operator, $\ran T$ is the range, and $\ker T$ is the kernel of $T$. 

\section{Useful results}
In this section, we note some useful results regarding property~$U$ and very smooth norms.

Recall that a Banach space $X$ is an \emph{Asplund space} if every separable subspace $Y$ of $X$ has a separable dual space $Y^*$. The following result (implicitly in \cite[Theorem~15]{SS}) is well known.

\begin{thm}\label{Asplund}
  A Banach space $X$ with property~$U$ in its bidual $X^{**}$ is an Asplund space.
\end{thm}

In addition, we will need the following known result for Asplund spaces.

\begin{thm} [see  {\cite[p.~112 and Theorem~8.3.3]{Fa}}] \label{eqv: m-basis wcg}
For a Banach space $X$, the following conditions are equivalent.
\begin{itemize}
    \item [(a)] $X$ has a shrinking \Marku\  basis, i.e., there are $(x_i)_{i\in I}$ in $X$ and $(f_i)_{i\in I}$ in $X^*$ such that 
$I$ has the cardinality $\dens X$,
and
        \begin{itemize}
            \item[$\bullet$] $f_i(x_j)=\delta_{ij}$,
            \item[$\bullet$] $\overline{\spann}(x_i)=X,$
            \item[$\bullet$] $\overline{\spann}(f_i)=X^*$.
        \end{itemize}
    \item [(b)] $X$ is weakly compactly generated (WCG) and Asplund.
\end{itemize}
\end{thm}


Recall that a Banach space $X$ is \emph{weakly compactly generated} (WCG) if $X$ is the closed linear span of some weakly compact subset of $X$. The most important result on WCG spaces is the following Amir--Lindenstrauss theorem.

\begin{thm}[see \cite{AL} or, e.g., {\cite[Theorem~1.2.5]{Fa}}] \label{Amir--Lindenstrauss}
A Banach space $X$ is weakly compactly generated if and only if there exist a set $\Gamma \neq \emptyset$ and an injective weak$^*$-to-weak continuous linear operator from $X^*$ to $c_0(\Gamma)$.
\end{thm}

We will also make use of the notion of very smooth spaces. First, recall that a Banach space $X$ is \emph{smooth} whenever for every $x\in S_X$, there exists a unique functional $f_x \in S_{X^*}$ such that $f_x(x)=1$. If $X$ is smooth, then the \emph{support mapping} $x \mapsto f_x$ from $S_X$ to $S_{X^*}$ is norm-to-weak$^*$ continuous (see, e.g., \cite[p.~22]{Di}).
\begin{defn} [see \cite{DF} or, e.g., {\cite[p.~31]{Di}}]
A smooth Banach space $X$ is called \emph{very smooth} if the support mapping $x \mapsto f_x$ from $S_X$ to $S_{X^*}$ is norm-to-weak continuous.
\end{defn}

It is well known that a Banach space $X$ is smooth whenever its dual space $X^*$ is strictly convex, but it need not be very smooth in general \cite{S}.

To prove that a renorming is very smooth, we will use the following result.

\begin{lem} [see \cite{G} or, e.g., {\cite[Lemma~III.2.14]{HWW}}]\label{lem: w and w* top agree}
A Banach space $X$ has property~$U$ in its bidual $X^{**}$ if and only if the relative weak and weak$^*$ topologies on $B_{X^*}$ coincide on $S_{X^*}$.
\end{lem}

\section{Renorming of Banach spaces with property~$U$}

In order to try to solve Problem~\ref{problem: U convex}, we follow the
strategy of the 
proof of the Observation.
This
proof 
has three steps.
\begin{itemize}

    \item [(i)] The Banach space $X$ has a shrinking \Marku\ basis (as proved by  Fabian and Godefroy \cite{FG}).
    \item[(ii)] Using the shrinking \Marku\ basis, one obtains an injective weak$^*$-to-weak continuous linear operator from $X^*$ to $c_0(\Gamma)$. This allows one to equip $X$ with a rather standard equivalent norm $|\cdot |$ such that for $\widetilde X:=(X,|\cdot |)$, the dual norm of $\widetilde X^*$ is strictly convex (as proved already by Amir and Lindenstrauss \cite{AL}).
    
    \item[(iii)] It can be shown that the renorming in (ii) is such that $\widetilde X$ is still an $M$-ideal in $\widetilde X^{**}$ (as proved by Harmand and Rao in \cite{HR}).
\end{itemize}

The step~(iii) can be extended from the $M$-ideal case (see \cite{HR} or \cite[Proposition~III.2.11]{HWW})
  to the property~$U$ case by the following theorem.

\begin{thm}\label{thm: eqv norm prop U}

Let $X$ be a Banach space with property~$U$ in its bidual $X^{**}$. If $Y$ is a Banach space and $T\colon Y \to  X$ a weakly compact operator, then 
\[
|x^*|:=\n{x^*}+\n{T^* x^*}, \quad x^* \in X^*,
\]
is an equivalent dual norm on $X^*$ for which
$\widetilde{X}:=(X,|\cdot|)$ has property~$U$ in $\widetilde{X}^{**}$. 
    
Moreover, if $T^*$ is injective and there is a strictly convex Banach space $Z$ such that $\ran T^* \subset Z\subset Y^*$, then $\widetilde{X}^*$ is strictly convex.
\end{thm}

\begin{proof}
Since $T^*$ is an adjoint operator, it is 
weak$^*$-to-weak$^*$ continuous, and thus the mapping $x^* \mapsto \n{T^*x^*}$, $x^* \in X^*$, is weak$^*$ lower semicontinuous. Therefore, $|\cdot|$ is an equivalent dual norm by a well-known result of Klee \cite{Kl1}; see, e.g., \cite[p.~106]{Di}. 

To calculate $|\cdot|$ on the third dual $X^{***}$, we use the following argument due to Harmand and Rao in \cite{HR}.
By definition, the operator
$$
S:(X^*,|\cdot|) \to  X^*\oplus_1 Y^*, \quad x^*\mapsto (x^*,T^*x^*),
$$
is isometric; hence
$$
S^{**}:(X^{***},|\cdot|) \to  X^{***}\oplus_1 Y^{***}
$$
is isometric too. One can easily see that $S^{**}x^{***}=(x^{***},T^{***}x^{***})$, so \[
|x^{***}|=\n{x^{***}}+\n{T^{***}x^{***}}, \quad x^{***}\in X^{***}.
\]
By the weak compactness
of $T^*$, we get that $\ran T^{***}\subset Y^*$, hence $\pi_{Y}T^{***}=T^{***}$. Since $\pi_Y T^{***}=T^{***}\pi_X$, we conclude that $T^{***}=T^{***}\pi_{X}$.

We need to show that the natural projection $\pi_{X}\in \mathcal{L}(X^{***})$ satisfies the condition \[
|\pi_{X}x^{***}|=|x^{***}| \quad \Rightarrow \quad \pi_{X}x^{***}=x^{***}.
\]
Let $x^{***}\in X^{***}$ be such that $|\pi_{X}x^{***}|=|x^{***}|$. Then
\begin{align*}
0&=|\pi_{X}x^{***}|-|x^{***}|\\
&=\n{\pi_{X}x^{***}}+\n{T^{***}\pi_{X}x^{***}}-\n{x^{***}}-\n{T^{***}x^{***}}\\
&=\n{\pi_{X}x^{***}}-\n{x^{***}}.
\end{align*}
Therefore, $\n{\pi_{X}x^{***}}=\n{x^{***}}$, and thus  $\pi_{X}x^{***}=x^{***}$ by property~$U$ of $X$ in $X^{**}$.

Moreover, if $T^*$ is injective and there is a strictly convex Banach space $Z$ such that $\ran T^* \subset Z\subset Y^*$, then, thanks to Klee's renorming theorem in \cite {Kl1} (see, e.g., \cite[Theorem~1, p.~100]{Di}), $\widetilde{X}^*$ is strictly convex. 
\end{proof}

Using Theorem~\ref{thm: eqv norm prop U}, we can now  give a partial answer to Problem~\ref{problem: U}.

\begin{thm}\label{main theorem prop U}
If a WCG Banach space $X$ has property~$U$ in its bidual $X^{**}$, then $X$ has a shrinking \Marku\ basis and $X$ admits an equivalent very smooth norm under which $X$ is totally smooth in its bidual.
\end{thm}

\begin{proof}
Since $X$ is also Asplund (see Theorem \ref{Asplund}), it has a shrinking \Marku\  basis (see Theorem \ref{eqv: m-basis wcg}). All we need to finish the proof is an injective weak$^*$-to-weak continuous linear operator $S\colon X^* \to c_0(\Gamma)$ (for some set $\Gamma$). Such an operator $S$ exists according to the Amir--Lindenstrauss theorem (see Theorem~\ref{Amir--Lindenstrauss}). However, $S$ can be very easily constructed using our shrinking \Marku\ basis. From now, we follow the proof of \cite[Theorem~III.4.6(e)]{HWW}.

Let $(x_i,f_i)_{i \in I}$ with $x_i \in X$, $f_i \in X^*$
be a shrinking \Marku\ basis.
Assuming $\n{x_i}=1$, we define an operator $S: X^* \to c_0(I)$ by
\[
x^* \mapsto \left (x^*(x_i) \right ), \quad x^*\in X^*.
\]
It is easy to check that the operator $S$ is well-defined, injective, and weak$^*$-to-weak continuous.
In particular, $S$ is weakly compact and weak$^*$-to-weak$^*$ continuous when considered as an operator into $c_0 (I)^{**}$, for which we use the notation~$\tilde S$. Hence, $\tilde S$ is the adjoint of a weakly compact operator $T: c_0(I)^{*} \to X$.

We now equip $c_0(I)$ with Day's equivalent strictly convex norm \cite{Da} (see, e.g., \cite[p.~94]{Di}).
Since $T$ satisfies the requirements in Theorem~\ref{thm: eqv norm prop U}, we get an equivalent smooth norm $|\cdot |$ on $X$
whose dual norm is strictly convex and for which $X$ still has property~$U$ in its bidual. The norm $|\cdot |$ on $X$ is, in fact, very smooth. Indeed, as was mentioned above,
the support mapping on a smooth space is always 
norm-to-weak$^*$ continuous, hence for $(X,|\cdot|)$, by Lemma~\ref{lem: w and w* top agree}, the support mapping is norm-to-weak continuous, i.e., $(X, |\cdot|)$ is very smooth.
\end{proof}


Since separable Banach spaces are WCG, Theorem~\ref{main theorem prop U}
gives an alternative proof to the separable case
considered by Sullivan in \cite{S} that was mentioned in the Introduction.

\begin{cor}
If a separable Banach space $X$ has property~$U$ in its bidual $X^{**}$, then $X$ admits an equivalent very smooth norm under which $X$ is totally smooth in its bidual.
\end{cor}

Readers particularly interested in the separable case as expounded in the previous corollary should notice that the rather easy argument of \cite[Proposition~1.f.3]{LT1} shows that a Banach space with a separable dual admits a shrinking \Marku\ basis. 

In order to obtain
examples of spaces having 
a shrinking \Marku\ basis, we can use the notion of $U^*$-spaces, which is dual to property~$U$.

\begin{defn}[see \cite{CN1}]
A Banach space $X$ is said to be a \emph{$U^*$-space} in its bidual $X^{**}$ if for every $x^{***} \in X^{***}$ with $\pi_X x^{***}\neq 0$, 
\[
\n{x^{***}-\pi_X x^{***}}<\n{x^{***}}.
\]
\end{defn}

In \cite[proof of Theorem~4.4]{CN1},
it was observed that the proofs of \cite[Theorems~1 and~3]{FG} essentially yield that every Asplund $U^*$-space has a shrinking \Marku\ basis and is WCG. Therefore, Theorem~\ref{main theorem prop U} applies to any $U^*$-space with property $U$ in its bidual. This, together with some help from
the literature, will be used in the next example.

\begin{ex}\label{ex: l_p sum c_0}
Let $\Gamma$ be an infinite set and let $1<p<\infty$. The $l_p$-sum $l_p(c_0(\Gamma))$ admits an equivalent very smooth norm under which $X$ is totally smooth in its bidual, but it cannot be equivalently renormed to be an $M$-ideal in its bidual. 
\end{ex}

\begin{proof}
It is well known (see, e.g., \cite[Example~III.1.4(a)]{HWW}) that $c_0(\Gamma)$ is an $M$-ideal in its bidual. Hence, clearly, it is a $U^*$-space. This property extends
to the $l_p$-sum, which also has property $U$ in its bidual (see \cite[Proposition~2.2]{CN1}). We have the desired renorming of $l_p(c_0(\Gamma))$ thanks to Theorem~\ref{main theorem prop U}.

Assume, for the sake of contradiction, that $X:=(l_p(c_0(\Gamma)),|\cdot |)$ is an $M$-ideal in $X^{**}$ for some equivalent norm $|\cdot |$ on $l_p(c_0(\Gamma))$. Since $M$-ideals in their biduals are stable by taking closed subspaces (see \cite{HL} or, e.g., \cite[Theorem III.1.6]{HWW}), $Y:=(l_p(c_0),|\cdot |)$ is an $M$-ideal in $Y^{**}$. This contradicts \cite[Proposition~4.4]{GKS} stating that if a separable $M$-ideal $Y$ in $Y^{**}$ has a boundedly complete Schauder decomposition $(Y_n)_{n=1}^\infty$, then all but finitely many subspaces $Y_n$ are reflexive. In our case, all $Y_n$, $n=1,2,\dots$,
are isomorphic to $c_0$ and thus non-reflexive.
\end{proof}

\section{Renorming of $(a,B,c)$-ideals}

According to \cite{GKS}, a closed subspace $X$ of a Banach space $Z$ is said to be an \emph{ideal} in $Z$ if there is a contractive
projection $P$ on $Z^*$ such that $\ker P =X^\perp$. In this case, the projection $P$ is called an \emph{ideal projection}. If $\ran P$ is norming, then the ideal is called \emph{strict}. If $Z=X^{**}$ and $P=\pi_X$, then 
the ideal is called \emph{canonical}. Canonical ideals are strict, but not vice versa.

Let $a,c\ge 0$ and let $B\subset \K$ be a compact set. If $X$ is an ideal in $Z$ with an ideal projection $P$ such that
\[
 \n{az^{*}+bPz^{*}}+c\n{Pz^{*}}\le \n{z^{*}} \quad \forall b \in B
 \]
for all $z^*$ in $Z^*$, then $X$ is said to be an \emph{$(a,B,c)$-ideal} in $Z$.
The $(a,B,c)$-ideals were introduced in \cite{O6} (see also \cite{O5}), but got their name later in \cite{OP}. This approach unifies all previously studied special cases of ideals. For instance, it is easy to see that $M$-ideals coincide with $(1,\{-1\},1)$-ideals, $u$-ideals coincide with $(1,\{-2\},0)$-ideals, and $h$-ideals are the same as $(1,\{-(1+\lambda)\colon \lambda \in S_{\mathcal{\C}}\},0)$-ideals. The notions $u$- and $h$-ideals have been deeply studied in \cite{GKS}.

In the context of $(a,B,c)$-ideals, the Observation (from the Introduction) triggers the following natural question (cf.\ Problem~\ref{problem: U}).

\begin{problem} \label{problem: aBc}
If $X$ is an 
$(a,B,c)$-ideal with property~$U$ in its bidual $X^{**}$, then does $X$ admit an equivalent norm under which $X$ is totally smooth and is still an $(a,B,c)$-ideal in its bidual?
\end{problem}

Similarly to the property~$U$ case, we prove that the renorming from step (iii) from the proof of the Observation can be extended to canonical $(a,B,c)$-ideals. Concerning the special case of $M$-ideals, recall that an $M$-ideal in its bidual is always canonical.

\begin{thm}\label{thm: eqv norm aBc}
Let $X$ be a Banach space which is a canonical $(a,B,c)$-ideal in its bidual $X^{**}$. If $Y$ is a Banach space and $T\colon Y \to  X$ is a weakly compact operator, then 
\[
|x^*|:=\n{x^*}+\n{T^* x^*}, \quad x^* \in X^*,
\]
is an equivalent dual norm under which $\widetilde{X}=(X,|\cdot |)$ is still a canonical $(a,B,c)$-ideal in $\widetilde{X}^{**}$.

Moreover, if $T^*$ is injective and there is a strictly convex Banach space $Z$ such that $\ran T^* \subset Z\subset Y^*$, then $\widetilde{X}^*$ is strictly convex.
\end{thm}

\begin{proof} 
The proof is essentially the same as the proof of Theorem~\ref{thm: eqv norm prop U}, except that here we need to show that $\widetilde{X}$ is a canonical $(a,B,c)$-ideal. This means that
\[
|ax^{***}+b\pi_X x^{***}|+c|\pi_X x^{***}|\le |x^{***}| \quad \forall b \in B
\]
holds for all $x^{***}$ in $X^{***}$.

Let $x^{***}\in X^{***}$. Then $\n{ax^{***}+b\pi_X x^{***}}+c\n{\pi_X x^{***}}\le \n{x^{***}}$ for all $b\in B$. Therefore, recalling that on $X^{***}$, the norm $|\cdot |$ is of the form $x^{***} \mapsto |x^{***}  |=\n{x^{***}}+\n{T^{***}x^{***}}$ and $T^{***}\pi_X =T^{***}$, we have
\begin{align*}
|ax^{***}+b\pi_X x^{***}|&+c|\pi_X x^{***}| \\
&= \n{ax^{***}+b\pi_X x^{***}}+\n{aT^{***}x^{***}+bT^{***}\pi_X x^{***}}\\
& \mbox{}\qquad +c\n{\pi_X x^{***}} +c\n{T^{***}\pi_X x^{***}}\\
&= \n{ax^{***}+b\pi_X x^{***}}+c\n{\pi_X x^{***}}\\
& \mbox{}\qquad+\n{aT^{***}x^{***}+bT^{***} x^{***}}+c\n{T^{***} x^{***}}\\
&\le \n{x^{***}}+|a+b|\n{T^{***}x^{***}}+c\n{T^{***}x^{***}}\\
&\le \n{x^{***}}+\n{T^{***}x^{***}}=|x^{***}|,
\end{align*}
because $|a+b|+c\le 1$ for every $b\in B$ (this can easily be verified by considering an arbitrary
$x^*\in S_{X^*}$).
\end{proof}

Our next result extends \cite[renorming result on p.~142 after Theorem~3]{FG} and \cite[Theorem~III.4.6(e)]{HWW} from $M$-ideals to strict $(a,B,c)$-ideals. Note that
Theorem~\ref{main theorem aBc}
applies in particular to $u$- and $h$-ideals.

\begin{thm}\label{main theorem aBc}
Let $a,c \ge 0$, and let $B$ be a compact set of scalars. Assume that a Banach space $X$ is a strict $(a,B,c)$-ideal in $X^{**}$. If $X$ has a shrinking \Marku\ basis, then $X$ admits an equivalent smooth norm
whose dual norm is strictly convex and under which $X$ becomes a canonical $(a,B,c)$-ideal in its bidual.
\end{thm}

The proof of Theorem~\ref{main theorem aBc} uses the following result that relies on \cite{GK} and extends \cite[Proposition~5.2, (1) and~(2)]{GKS}, where $u$- and $h$-ideals were considered. 

\begin{prop}\label{strict and canonical}
Let a Banach space $X$ be a strict $(a,B,c)$-ideal in $X^{**}$. If $X$ does not contain $l_1$ isomorphically, then $X$ is a canonical $(a,B,c)$-ideal in $X^{**}$.
\end{prop}
\begin{proof}
The proof follows the scheme of the proof of \cite[Proposition~5.2, (1) and~(2)]{GKS}.

By assumption, there is an $(a,B,c)$-ideal projection $P$ on $X^{***}$ such that $\ker P =X^\perp$ and $\ran P$ is a norming subspace of $X^{***}$. It suffices to show that $\ran P =X^*$, because then $P=\pi_X$ (recall that $\ker \pi_X=X^\perp$ and $\ran \pi_X=X^*)$, meaning that $X$ is a canonical $(a,B,c)$-ideal in $X^{**}$.

Since $X$ does not contain $l_1$ isomorphically, we get from \cite[Corollary~5.5]{GK} that $X^{***}$ contains a minimal norming subspace, which is, by definition, the intersection of all norming subspaces of $X^{***}$. We know that $X^*$ is norming in $X^{***}$. On the other hand, a proper subspace $U$ of $X^*$ cannot be norming in $X^{***}$. Indeed, assume, for the sake of contradiction, that such a $U$ is norming in $X^{***}$. By the Hahn--Banach theorem, there is $x^{**} \in X^{**}$ such that $x^{**}(u)=0$ for every $u\in U$, but $x^{**}\neq 0$. Since $U$ is norming in $X^{***}$, we get
\[
\n{x^{**}}=\sup_{u\in B_U}|x^{**}(u)|=0,
\]
which is a contradiction.

Therefore, $X^*$ is the minimal norming subspace, and thus $X^*\subset \ran P$. Since now also $\ran P \subset X^*$ (indeed, if $x^{***}=x^* +x^\perp \in X^{***}$ with $x^{*}\in X^{*}$, $x^\perp \in X^\perp$, then $Px^{***}=Px^*=x^*$, because $X^*\subset \ran P$), we have $\ran P=X^*$, as desired.
%
\end{proof}

\begin{proof}[Proof of Theorem~\ref{main theorem aBc}]

Since $X$ has a shrinking \Marku\ basis, by Theorem~\ref{eqv: m-basis wcg}, $X$ is Asplund. An Asplund space cannot contain $l_1$ isomorphically, and thus, by Proposition~\ref{strict and canonical}, $X$ is a canonical $(a,B,c)$-ideal.


Now, using Theorem~\ref{thm: eqv norm aBc} (instead of Theorem~\ref{thm: eqv norm prop U}) in the proof of Theorem~\ref{main theorem prop U}, this immediately yields an equivalent smooth norm $|\cdot |$ on $X$ whose dual norm is strictly convex and for which  $X$ is still a canonical $(a,B,c)$-ideal in its bidual.
\end{proof}

If the latter proof is carried out under the supplementary assumption of property $U$, then one obtains the following partial positive answer to Problem~\ref{problem: aBc}, which is quite similar to Theorem~\ref{main theorem prop U}.

\begin{thm}\label{thm: U and aBc}
Let $a,c\ge0$, and let $B$ be a compact set of scalars.
Assume that a Banach space $X$ is a strict $(a,B,c)$-ideal with property~$U$ in $X^{**}$. If $X$ is WCG, then $X$ has a shrinking \Marku\ basis and $X$ admits an equivalent very smooth norm under which $X$ becomes a totally smooth canonical $(a,B,c)$-ideal in its bidual.
\end{thm}

For separable spaces, we can again omit the WCG-assumption.

\begin{cor}
Let $a,c \ge 0$, and let $B$ be a compact set of scalars.
Assume that a separable Banach space $X$ is a strict $(a,B,c)$-ideal with property~$U$ in $X^{**}$. Then $X$ admits an equivalent very smooth norm under which $X$ is a totally smooth canonical $(a,B,c)$-ideal in its bidual.
\end{cor}

As it was recalled in Section~3, every Asplund $U^*$-space has a shrinking \Marku\ basis. Since this property is preserved under isomorphisms, Theorem~\ref{thm: U and aBc} immediately implies the following.

\begin{cor} \label{cor: prop U and isom to U*}
Let $a,c\ge 0$
and let $B$ be a compact set of scalars. Assume that a Banach space $X$ is a strict $(a,B,c)$-ideal with property~$U$ in $X^{**}$. If $X$ is isomorphic to a $U^*$-space, then $X$ admits an equivalent very smooth norm under which $X$ is a totally smooth canonical $(a,B,c)$-ideal in its bidual.
\end{cor}

Similarly to Theorem~\ref{thm: U and aBc}, we can apply Theorem~\ref{main theorem aBc} to obtain the following result. However, here we need to use a couple of auxiliary results from the literature.

\begin{thm}\label{thm: isom to U*}
\begin{sloppypar}Let a Banach space $X$ be a strict $(a,B,c)$-ideal in $X^{**}$ with {$\max\{|b|\colon b \in B\}+c>1$}. If $X$ is isomorphic to a $U^*$-space, then $X$ admits an equivalent smooth norm
whose dual norm is strictly convex and under which $X$ becomes a canonical $(a,B,c)$-ideal in its bidual.
\end{sloppypar}
\end{thm}

\begin{proof}
A $U^*$-space does not contain $l_1$ isomorphically. This fact was observed in \cite[Proposition~4.1]{CN1} as a direct consequence of \cite[Proposition~2.6]{GKS}. Hence, $X$ does not contain $l_1$ isomorphically, and therefore, by Proposition~\ref{strict and canonical}, our $(a,B,c)$-ideal $X$ is canonical. But canonical $(a,B,c)$-ideals with $B$ and $c$ as above
are Asplund spaces (see \cite[proof of Theorem~4.1]{O6} or \cite[Lemma~4.2]{OZ}). Hence, $X$ has a shrinking \Marku\ basis, and Theorem~\ref{main theorem aBc} applies.
\end{proof}

Note that the above $(a,B,c)$-ideal assumption is satisfied in all important cases, including $M$-, $u$-, and $h$-ideals.

Clearly, $M$-ideals in their biduals and, more generally, the canonical $(1,\{-1\},c)$-ideals with $c\in (0,1]$ are $U^*$-spaces. Hence, from Theorem~\ref{thm: isom to U*}, we have the following example. 

\begin{ex}\label{ex:(1,1,c)}
Let $X$ be a canonical $(1,\{-1\},c)$-ideal in $X^{**}$ with $c\in (0,1]$. Then $X$ admits an equivalent smooth norm whose dual norm is strictly convex and under which $X$ is a canonical $(1,\{-1\},c)$-ideal in its bidual.
\end{ex}
\begin{rem}
A particular example of a $(1,\{-1\},c)$-ideal is provided by certain
renormings of the James space~$J$, as shown in \cite[Example~3.5]{CN1}.
Namely, for $\delta>\sqrt2$ the renorming $J_\delta$ of the James space in
\cite[Example~3.5]{CN1} is a canonical $(1,\{-1\},c)$-ideal in $J_\delta^{**}$ if
\[
\max \Bigl\{ \frac{(1+\delta c)^2}{\delta^2},
\frac{(1+\delta c)^2+ (1+ c)^2 + 2(\delta c)^2}{2\delta^2} \Bigr\}
<\frac12.
\]
(We take this opportunity to point out a disturbing typo in \cite{CN1}
where the denominator in the first item of the maximum is $2$ instead of
$\delta^2$.)
It is easy to see that for each $c<1/\sqrt3$, there is some
$\delta>\sqrt2$ satisfying the above inequality.

The point of this remark is that the James space~$J$ cannot be renormed to be
an $M$-ideal in its bidual, since it is  non-reflexive and its bidual is
separable, being isomorphic to $J$. However, a non-reflexive space that is
an $M$-ideal in its bidual contains a copy of $c_0$ (see \cite{HL} or, e.g, \cite[Corollary~III.3.7(a)]{HWW};
hence the renormed James space $J_\delta$ is a nontrivial instance of
Example~\ref{ex:(1,1,c)}.
\end{rem}

Let us conclude the paper with a couple of examples where Corollary~\ref{cor: prop U and isom to U*} applies.

It is known (see \cite[Proposition~2.2]{CN1}) that the space in Example~\ref{ex: l_p sum c_0} is a canonical $u$-ideal in its bidual. Using Corollary~\ref{cor: prop U and isom to U*} (instead of Theorem~\ref{main theorem prop U}) in the proof of Example~\ref{ex: l_p sum c_0} allows us to strengthen this -- using the same norm as in Example~\ref{ex: l_p sum c_0} --
as follows. Note that canonical $u$-ideals in their biduals could be considered as the closest important weakenings of $M$-ideals in their biduals.

\begin{ex}
Let $\Gamma$ be an infinite set and let $1<p<\infty$. The $l_p$-sum $l_p(c_0(\Gamma))$ admits an equivalent very smooth norm under which it is a totally smooth canonical $u$-ideal in its bidual. But $l_p(c_0(\Gamma))$  cannot be equivalently renormed to be an $M$-ideal in its bidual. 
\end{ex}

Our last Example~\ref{ex: last} will concern a large class of spaces. We need some more notation. For Banach spaces $X$ and $Y$, we denote by $\mathcal{L}(X,Y)$ the Banach space of bounded linear operators from $X$ to $Y$, and by $\mathcal{K}(X,Y)$ its subspace of compact operators. We write $\mathcal{L}(X)$ and $\mathcal{K}(X)$, respectively, if $X=Y$. Recall that a net $(K_\alpha)$ in $\mathcal{K}(X)$ is a \emph{compact approximation of the identity} (CAI) provided $K_\alpha x\to x$ for all $x \in X$. If, moreover, $K_\alpha ^* x^*\to x^*$ for all $x^* \in X^*$, then $(K_\alpha)$ is called a \emph{shrinking CAI}. If $X$ has a CAI such that the convergence is uniform on compact subsets of $X$, then $X$ is said to have the \emph{compact approximation property} (CAP).

Let $1<p<\infty$. A Banach space $X$ is said to have the \emph{upper $p$-property} (cf. \cite[Proposition~1.1]{OW} or \cite[pp.~306 and 327]{HWW}) if $X$ admits a shrinking CAI $(K_\alpha)$ such that 
\[
\limsup_\alpha \sup_{x,y \in B_X} \n{K_\alpha x +(y-K_\alpha y)}\le (\n{x}^p +\n{y}^p)
^{1/p}.
\]

If $X$ and $Y$ are both reflexive Banach spaces, and $X$ or $Y$ has the CAP, then $\mathcal{K}(X,Y)^{**}=\mathcal{L}(X,Y)$ (see \cite[Corollary~1.3]{GS}; in the AP-case, this is a well-known result due to Grothendieck).
There is a vast literature studying the position of $\mathcal{K}(X,Y)$ in $\mathcal{L}(X,Y)$ in terms of ideals (see, e.g., \cite{HJO} for recent results and a large set of references). We are going to use \cite[Theorem~4.4 and Corollary~4.5]{CN2}, from which one can see that $\mathcal{K}(X,Y)$  is a $U^*$-space with property $U$ and also a strict $(a,\{-a\},c)$-ideal for all $a,c>0$ such that $a^p+c^p\le 1$ in  $\mathcal{L}(X,Y)$, whenever $X$ is an arbitrary Banach space and $Y$ is a Banach space having the upper $p$-property. Thanks to Corollary~\ref{cor: prop U and isom to U*}, we have the following rather general example.

\begin{ex}\label{ex: last}
Let $X$ and $Y$ be reflexive Banach spaces and let $1<p<\infty$. If $Y$ has the upper $p$-property, then $\mathcal{K}(X,Y)$ admits an equivalent very smooth norm under which $\mathcal{K}(X,Y)$ is a totally smooth canonical $(a,\{-a\},c)$-ideal in its bidual for all $a,c>0$ such that $a^p+c^p\le 1$.
\end{ex}

Besides the $l_p(\Gamma)$ spaces and the Lorentz sequence spaces $d(w,p)$, there are many reflexive spaces enjoying the upper $p$-property, as one can see from \cite[pp.~306,~327]{HWW}. In the context of Example~\ref{ex: last}, we are interested in cases when $\mathcal{K}(X,Y)$ is \emph{not} an $M$-ideal in its bidual $\mathcal{L}(X,Y)$. Here the classical example, due to Hennefeld (see \cite{H} or, e.g., \cite[p.~305]{HWW}), is that $\mathcal{K}(d(w,p))$ is \emph{not} an $M$-ideal in its bidual $\mathcal{L}(d(w,p))$. If $1<p\le q<\infty$, then $\mathcal{K}(l_p(\Gamma), d(w,q))$ \emph{is} an $M$-ideal in $\mathcal{L}(l_p(\Gamma),d(w,q))$ by \cite{O1} (see, e.g., \cite[p.~53 or p.~66]{O4}). However, in \cite{O3} (see, e.g., \cite[p.~73 and p.~77]{O4}), it is proved that for $1<q<p<\infty$ and an infinite set $\Gamma$, $\mathcal{K}(l_p(\Gamma),d(w,q))$ is \emph{not} an $M$-ideal in $\mathcal{L}(l_p(\Gamma),d(w,q))$ whenever $w\in l_{\frac{p}{p-q}}$, neither is $\mathcal{K}(d(v,p)^*,d(w,q))$ in $\mathcal{L}(d(v,p)^*,d(w,q))$, $1<p,q<\infty$, whenever $p>(p-1)q$ and $d(v,p)^*$ (which is a sequence space) is contained as a linear subspace in $d(w,q)$.

\subsection*{Acknowledgements}

This research was partially supported by institutional research funding IUT20-57 of the Estonian Ministry of Education and Research and by the German Academic Exchange Service (DAAD).

\bibliography{HWW}

\end{document}